\documentclass{article}
\usepackage[]{KLMM/klmm}
\usepackage{listings}

\DeclareMathOperator{\ord}{ord}
\newcommand{\Fp}{\mathbb{F}_p}

\newcommand{\res}[2]{[#1]_{#2}}
\newcommand{\Dc}{\mathcal{D}}
\newcommand{\ppart}{\operatorname{ppart}}

\lstdefinestyle{lean}{
  basicstyle=\small\ttfamily,
  columns=fullflexible,
  keepspaces=true,
  breaklines=true,
  frame=single,
  framerule=0.3pt,
  rulecolor=\color{black!30},
  xleftmargin=0.6em,
  framexleftmargin=0.6em,
  aboveskip=0.8em,
  belowskip=0.8em
}

\title{A Greatest Common Divisor Criterion of Certain Binomial Coefficients}
\author{Dakai~Guo\institution{State Key Laboratory of Mathematical Sciences,
Academy of Mathematics and Systems Science, Chinese Academy of Sciences, and
University of Chinese Academy of Sciences. 
This work is supported by the Strategic Priority
Research Program of Chinese Academy of Sciences under Grant XDA0480502 and
XDA0480503.}
    \and Ruichen Qiu\instref{1}
    \and Yichuan Cao\instref{1}
    \and Ruyong Feng\instref{1}
    \and Xiao-Shan Gao\instref{1}
}
\date{\today}

\begin{document}
\maketitle

\begin{abstract}
The binomial greatest common divisor (gcd) criterion recorded as OEIS A080170 is proven. The criterion also appears as conjecture (17) in Ralf Stephan's list of OEIS conjectures.  For
$k\geq 2$, put
\[
  D(k)=\gcd_{2\leq q\leq k+1}\binom{qk}{k},
  \qquad n=k+1.
\]
If $P$ is the largest prime-power component $p^a$ exactly dividing $n$, then the criterion asserts
\[
  D(k)=1 \quad\Longleftrightarrow\quad \frac{n}{P}>P.
\]
The proof is formalized in Lean and
the Lean artifact is accepted as part of the Formal Conjectures project.
Both the natural-language proof and the Lean formalization are generated by the
MechMath Agent Team, an AI agent developed by the authors.
\end{abstract}



\section{Introduction}

Let
\[
  D(k)=\gcd_{2\leq q\leq k+1}\binom{qk}{k}
\]
for $k\geq 2$.  The On-Line Encyclopedia of Integer Sequences (OEIS) is a database
of integer sequences, together with descriptions, formulae, programs,
references, and links \cite{sloane}.  OEIS A080170 records the values of
$k$ for which $D(k)=1$, together with a conjectural description in terms of
the factorization of $k+1$ \cite{oeis}.  This binomial gcd criterion was
formulated by Ralf Stephan as conjecture (17) in his collection, \emph{Prove or Disprove: 100 Conjectures from the OEIS} \cite{stephan}.  The
purpose of this note is to prove that description.

The theorem fits into a broader line of work on common divisors of selected
binomial coefficients.  A classical starting point is Ram's theorem that
\[
  \gcd_{0<i<n}\binom{n}{i}
\]
is $p$ if $n$ is a power of the prime $p$, and is $1$ otherwise
\cite{ram}.  The same row-gcd criterion appears, for example, in the work of
Brown and Peterson and in the modern formulation of L\"u and Panov
\cite{brownpeterson,lupanov}.  Other work changes the selected set of
binomial coefficients: Joris, Oestreicher, and Steinig studied gcds of certain
finite regions of Pascal's triangle \cite{jorisoestreichersteinig}; McTague
studied gcds of the subfamily
$\binom{n}{q},\binom{n}{2q},\binom{n}{3q},\ldots$ \cite{mctague}, and Hong
studied further families selected by arithmetic restrictions on the lower
index \cite{hong}.  The present problem is different: the lower entry is
fixed at $k$, while the upper entry varies through the multiples $qk$.

For a positive integer $n$ and a prime $p$, write $p^a\Vert n$ when $p^a\mid n$ and
$p^{a+1}\nmid n$.  We call such $p^a$ an exact prime-power component of $n$,
and define
\[
  \ppart(n)=\max\{p^a:\ p^a\Vert n\}
\]
for $n>1$.  The main result is the following theorem, which proves OEIS A080170.

\begin{theorem}
\label{thm:a080170}
Let $k\geq 2$ and put $n=k+1$.  Then
\[
  \gcd_{2\leq q\leq k+1}\binom{qk}{k}=1
  \quad\Longleftrightarrow\quad
  \frac{n}{\ppart(n)}>\ppart(n).
\]
\end{theorem}
The proof first shows, by a finite-difference argument, that every prime divisor of $D(k)$ divides $k+1$.  For primes $p\mid k+1$, Lucas' theorem reduces the problem to a digitwise stabilizer statement for a finite box of base-$p$ digits.  The resulting $p$-primary criterion is $p\mid D(k)$ if and only if $n/p^a\leq p^a$, and the theorem follows by checking the largest exact prime-power component of $n$.  

This problem also has a formalization motivation.  
Problem A080170 is not only an OEIS
problem from Stephan's collection, but also a member of
\texttt{FC100OpenSet1}, a distinguished subset of 100 open research problems
in the Formal Conjectures benchmark.  
The update associated with this proof
therefore changes the category count for that subset from 96 open and 4
solved entries to 95 open and 5 solved entries. Thus, the result provides a
test case in which a previously open research-level formal statement receives
both a mathematical proof and a machine-checkable Lean proof.

Formal Conjectures is an open
Lean 4 repository and benchmark of formalized mathematical conjectures
\cite{formalconjecturesrepo,formalconjecturespaper}.  The project paper
reports a benchmark with 2615 formalized mathematical problem statements,
including 1029 open research conjectures and 836 solved problems for proof
autoformalization \cite{formalconjecturespaper}.  Its stated goals include
providing benchmark problems for automated theorem proving and
autoformalization, clarifying the exact meaning of conjectures through
formalization, and identifying definitions that should enter mathlib
\cite{formalconjecturesrepo}.

Formal Conjectures is an open benchmark for verified mathematical discovery
\cite{formalconjecturespaper,formalconjecturesrepo}.  The present theorem is
one of the OEIS problems selected for that benchmark, originating from
Stephan's list of one hundred OEIS conjectures \cite{stephan}.  More
specifically, \texttt{gcdCondition\_iff\_primePowerCondition}, in the
\texttt{OeisA80170} namespace, is listed in the subset file
\texttt{FC100OpenSet1.lean}.  The associated
     pull request changes its tag
from \texttt{research open} to
\texttt{research solved} and updates the verified category counts for that
100-problem subset from \(96\) open and \(4\) solved entries to \(95\) open
and \(5\) solved entries. Thus, the Lean development turns the A080170 entry
from an open formal-conjecture instance into a solved one.

Both the natural-language proof and the Lean formalization are generated by the
MechMath Agent Team developed by the authors~\cite{mechmath}.
MechMath Agent Team is a large language model driven agent designed to generate proofs for mathematical theorems expressed in both natural language and formal language in Lean. 

The rest of the paper is organized as follows.
In Section \ref{sec-proof}, the proof of Theorem \ref{thm:a080170} is given.
In Section \ref{sec:lean}, the Lean formalization of the proof is described.
In Section \ref{sec-conc}, conclusions are presented..

\section{Proof of Theorem \ref{thm:a080170}}
\label{sec-proof}
In this section, the proof of Theorem \ref{thm:a080170} is presented in several steps.  
The only standard external input 
is Lucas's theorem
\cite{lucas}; for broader background on Lucas-type congruences, see the
survey \cite{mestrovic}. For completeness, we include the short polynomial
proof of the non-vanishing criterion used here.

\subsection{Preliminaries}

For an integer $x$ and a positive integer $M$, let $\res{x}{M}$ denote the
least nonnegative residue of $x$ modulo $M$.

\begin{lemma}[Newton interpolation by forward differences]
\label{lem:newton}
Let $f$ be a polynomial of degree at most $d$, and define
\[
  \Delta f(x)=f(x+1)-f(x).
\]
For every integer $a$ and every $x$,
\[
  f(x)=\sum_{r=0}^{d}\binom{x-a}{r}\Delta^r f(a).
\]
\end{lemma}

\begin{proof}
Let
\[
  g(x)=\sum_{r=0}^{d}\binom{x-a}{r}\Delta^r f(a).
\]
Both $f$ and $g$ have degree at most $d$.  It is enough to show that they
agree at $x=a+j$ for $j=0,\ldots,d$.

For $r\geq 0$,
\[
  \Delta^r f(a)=\sum_{i=0}^{r}(-1)^{r-i}\binom{r}{i}f(a+i).
\]
Thus, at $x=a+j$,
\[
  g(a+j)
  =\sum_{r=0}^{j}\binom{j}{r}
       \sum_{i=0}^{r}(-1)^{r-i}\binom{r}{i}f(a+i).
\]
The coefficient of $f(a+i)$ in this expression is
\[
  \sum_{r=i}^{j}\binom{j}{r}(-1)^{r-i}\binom{r}{i}
  =\binom{j}{i}\sum_{u=0}^{j-i}(-1)^u\binom{j-i}{u}.
\]
This coefficient is $1$ if $i=j$ and $0$ otherwise.  Therefore
$g(a+j)=f(a+j)$ for $j=0,\ldots,d$.  Hence $f=g$.
\end{proof}

\begin{lemma}[Lucas non-vanishing criterion]
\label{lem:lucas}
Let $p$ be prime, and write
\[
  N=\sum_i N_i p^i,\qquad K=\sum_i K_i p^i
\]
with $0\leq N_i,K_i<p$.  Then
\[
  \binom{N}{K}\not\equiv 0 \pmod p
  \quad\Longleftrightarrow\quad
  K_i\leq N_i \text{ for every } i.
\]
\end{lemma}

\begin{proof}
Work in the polynomial ring $\Fp[X]$.  Since $(1+X)^p=1+X^p$ in
characteristic $p$,
\[
  (1+X)^N
  =\prod_i (1+X)^{N_i p^i}
  =\prod_i (1+X^{p^i})^{N_i}.
\]
The coefficient of $X^K$ in the last product is
\[
  \prod_i \binom{N_i}{K_i},
\]
where a factor is interpreted as $0$ if $K_i>N_i$.  Each nonzero factor is a
nonzero element of $\Fp$, because $0\leq K_i\leq N_i<p$.  Therefore the
coefficient, which is $\binom{N}{K}\bmod p$, is nonzero if and only if
$K_i\leq N_i$ for all $i$.
\end{proof}

\subsection{Primes in the gcd}

\begin{lemma}
\label{lem:prime-divides-n}
Let $k\geq 1$ and
\[
  D(k)=\gcd_{2\leq q\leq k+1}\binom{qk}{k}.
\]
If a prime $p$ divides $D(k)$, then $p\mid k+1$.
\end{lemma}

\begin{proof}
Set
\[
  F(x)=\binom{kx}{k}.
\]
This is a polynomial in $x$ of degree $k$, and
\[
  F(0)=0,\qquad F(1)=1,\qquad F(q)=\binom{qk}{k}.
\]
Assume that $p$ divides $D(k)$.  Then
\[
  F(2)\equiv F(3)\equiv\cdots\equiv F(k+1)\equiv 0 \pmod p.
\]
For $r=0,\ldots,k$,
\[
  \Delta^rF(1)
  =\sum_{j=0}^{r}(-1)^{r-j}\binom{r}{j}F(1+j).
\]
Modulo $p$, all terms with $j\geq 1$ vanish, while $F(1)=1$.  Hence
\[
  \Delta^rF(1)\equiv (-1)^r \pmod p.
\]
By Lemma~\ref{lem:newton}, evaluated at $x=0$ and $a=1$,
\[
  F(0)=\sum_{r=0}^{k}\binom{-1}{r}\Delta^rF(1).
\]
Since $\binom{-1}{r}=(-1)^r$, we obtain
\[
  0=F(0)\equiv
  \sum_{r=0}^{k}(-1)^r(-1)^r
  =k+1
  \pmod p.
\]
Thus $p\mid k+1$.
\end{proof}

\subsection{Digit boxes}

Fix a prime $p$ and an integer $L\geq 1$.  Put $P=p^L$.  If
$0\leq c<P$ has base-$p$ expansion
\[
  c=\sum_{i=0}^{L-1}c_i p^i,\qquad 0\leq c_i<p,
\]
define the digit box
\[
  \Dc_c=\left\{\sum_{i=0}^{L-1}y_i p^i:\ 0\leq y_i\leq c_i
  \text{ for every }i\right\}.
\]
Thus $0,c\in\Dc_c$, and every element of $\Dc_c$ lies in $[0,c]$.

\begin{lemma}[Gap bound]
\label{lem:gap}
Let $P=p^L$ and $0\leq c<P$.  If the elements of $\Dc_c$ are listed in
increasing order, every gap between two consecutive elements is at most
$P/p$.
\end{lemma}

\begin{proof}
We argue by induction on $L$.  For $L=1$, the set is
$\{0,1,\ldots,c\}$, so all gaps are $1=P/p$.

Assume $L>1$ and write $P=pH$, where $H=p^{L-1}$.  Write
\[
  c=dH+r,\qquad 0\leq d<p,\quad 0\leq r<H.
\]
Then
\[
  \Dc_c=\bigcup_{a=0}^{d}(aH+\Dc_r).
\]
Inside each block $aH+\Dc_r$, the induction hypothesis gives gaps at most
$H/p\leq H=P/p$.  The last element of $aH+\Dc_r$ is $aH+r$, and the first
element of the next block is $(a+1)H$.  The gap between them is
\[
  (a+1)H-(aH+r)=H-r\leq H=P/p.
\]
Therefore all consecutive gaps are at most $P/p$.
\end{proof}

\begin{lemma}[No nonzero translation]
\label{lem:translation}
Let $P=p^L$ and $0\leq c<P-P/p$.  Suppose $0\leq S\leq c$ and
\[
  \res{z+S}{P}\in\{0,1,\ldots,c\}
  \qquad\text{for every }z\in\Dc_c.
\]
Then $S=0$.
\end{lemma}

\begin{proof}
Assume $S>0$.

If $S<P-c$, then $c\in\Dc_c$ gives
\[
  \res{c+S}{P}=c+S\in\{c+1,\ldots,P-1\},
\]
contradicting the hypothesis.

It remains to consider $S\geq P-c$.  Since $S\leq c$, the open interval
\[
  I=(c-S,\ P-S)
\]
is contained in $[0,c]$.  Its length is
\[
  (P-S)-(c-S)=P-c>P/p.
\]
If $I$ contained no element of $\Dc_c$, then the consecutive elements of
$\Dc_c$ immediately surrounding $I$ would have a gap of length greater than
$P/p$, contradicting Lemma~\ref{lem:gap}.  Hence there is some
$z\in\Dc_c$ with
\[
  c-S<z<P-S.
\]
Then
\[
  c<z+S<P,
\]
so $\res{z+S}{P}\notin\{0,1,\ldots,c\}$, again contradicting the hypothesis.
Thus $S=0$.
\end{proof}

\begin{theorem}[Digit-box stabilizer]
\label{thm:stabilizer}
Let $P=p^L$ and $0<c<P-P/p$.  Let $s$ be coprime to $p$.  Then
\[
  \res{s y}{P}\in\{1,\ldots,c\}
  \quad\text{for every }y\in\Dc_c\setminus\{0\}
\]
if and only if
\[
  s\equiv 1 \pmod {p^{L-\ord_p(c)}}.
\]
\end{theorem}

\begin{proof}
First assume $s\equiv 1\pmod {p^{L-\ord_p(c)}}$.  Put $v=\ord_p(c)$.  The
lowest $v$ base-$p$ digits of $c$ are $0$, so every $y\in\Dc_c$ is divisible
by $p^v$.  Hence $(s-1)y$ is divisible by $p^{L-v}p^v=P$, and
\[
  \res{s y}{P}=y.
\]
For $y\neq 0$ this lies in $\{1,\ldots,c\}$.

Conversely, assume
\[
  \res{s y}{P}\in\{1,\ldots,c\}
  \qquad(y\in\Dc_c\setminus\{0\}).
\]
We prove the desired congruence by induction on $L$.

If $L=1$, then $P=p$ and $0<c<p-1$.  For $p=2$ there is no such $c$.
For odd $p$, $\Dc_c=\{0,1,\ldots,c\}$.  Multiplication by $s$ is injective on
the nonzero residues modulo $p$, so the image of $\{1,\ldots,c\}$ is exactly
$\{1,\ldots,c\}$.  Taking sums modulo $p$ gives
\[
  s\,\frac{c(c+1)}{2}\equiv \frac{c(c+1)}{2}\pmod p.
\]
Because $0<c<p-1$, the factor $c(c+1)/2$ is nonzero modulo $p$, so
$s\equiv 1\pmod p$.  This is the asserted congruence for $L=1$.

Now let $L>1$ and write $P=pP'$, where $P'=p^{L-1}$.  Decompose
\[
  c=c_0+pC,\qquad s=s_0+pS,
\]
with $0\leq c_0<p$, $1\leq s_0<p$, and $0\leq S<P'$.

If $c_0=0$, then $c=pC$ and
\[
  \Dc_c=p\Dc_C.
\]
For every nonzero $z\in\Dc_C$, the element $pz$ belongs to
$\Dc_c\setminus\{0\}$.  The hypothesis gives
\[
  \res{s(pz)}{P}=p\,\res{s z}{P'}\in\{1,\ldots,pC\}.
\]
Thus
\[
  \res{s z}{P'}\in\{1,\ldots,C\}
  \qquad(z\in\Dc_C\setminus\{0\}).
\]
Also $C<P'-P'/p$, since $pC<P-P/p=(p-1)P'$.  By induction,
\[
  s\equiv 1\pmod {p^{(L-1)-\ord_p(C)}}.
\]
As $\ord_p(c)=1+\ord_p(C)$, this is exactly
\[
  s\equiv 1\pmod {p^{L-\ord_p(c)}}.
\]

It remains to handle $c_0>0$.  Then $\ord_p(c)=0$, so we must prove
$s\equiv 1\pmod P$.

First consider elements $pz$ with $z\in\Dc_C\setminus\{0\}$.  The same
division by $p$ as above gives
\[
  \res{s z}{P'}\in\{1,\ldots,C\}
  \qquad(z\in\Dc_C\setminus\{0\}).
\]
If $C>0$, induction and the already-proved forward direction imply
\[
  \res{s z}{P'}=z\qquad(z\in\Dc_C).
\]
If $C=0$, this conclusion is also true, because $\Dc_C=\{0\}$.

Now fix $a\in\{1,\ldots,c_0\}$ and $z\in\Dc_C$.  The element
$a+pz$ lies in $\Dc_c\setminus\{0\}$.  Write
\[
  s_0a=r_a+p e_a,\qquad 0\leq r_a<p.
\]
Using $\res{s z}{P'}=z$, the quotient after the lowest base-$p$ digit of
$\res{s(a+pz)}{P}$ is
\[
  \res{z+Sa+e_a}{P'}.
\]
Since $\res{s(a+pz)}{P}\leq c_0+pC$, this quotient is at most $C$.  Hence
\[
  \res{z+Sa+e_a}{P'}\in\{0,\ldots,C\}
  \qquad(z\in\Dc_C).
\]
Putting $z=0$ shows that
\[
  R_a:=\res{Sa+e_a}{P'}\in\{0,\ldots,C\},
\]
and the last display becomes
\[
  \res{z+R_a}{P'}\in\{0,\ldots,C\}
  \qquad(z\in\Dc_C).
\]
By Lemma~\ref{lem:translation}, applied with modulus $P'$ and digit bound
$C$, we have $R_a=0$.  Therefore
\[
  Sa+e_a\equiv 0\pmod {P'}
  \qquad(1\leq a\leq c_0).
\]
For $a=1$ we have $e_1=0$, since $s_0<p$.  Thus $S\equiv 0\pmod {P'}$.
Because $0\leq S<P'$, it follows that $S=0$.  Then the congruence above gives
$e_a\equiv 0\pmod {P'}$ for every $a$.  Since $0\leq e_a<p\leq P'$, we get
$e_a=0$ for all $a=1,\ldots,c_0$.  In particular
\[
  s_0c_0<p.
\]

Finally take $a=c_0$ and $z=C$, so that $a+pz=c$.  We have $S=0$, hence
$s=s_0$, and $\res{s C}{P'}=C$.  Since $s_0c_0<p$, there is no carry from the
lowest digit, so
\[
  \res{s c}{P}=s_0c_0+pC.
\]
The hypothesis says this is at most $c=c_0+pC$.  Therefore
$s_0c_0\leq c_0$.  Because $c_0>0$, we conclude $s_0=1$.  Together with
$S=0$, this gives $s=1$, i.e. $s\equiv 1\pmod P$.  This completes the
induction and the proof.
\end{proof}

\subsection{The zero-run lemma}

\begin{lemma}[Zero-run lemma]
\label{lem:zero-run}
Let $p$ be prime, let $m\geq 1$, and let $M$ be coprime to $p$.  Assume
$p\nmid m+1$.  Let $P=p^L$ be the least power of $p$ satisfying $m<P$.
If
\[
  \binom{m+tM}{m}\equiv 0\pmod p
  \qquad(t=1,\ldots,m),
\]
then
\[
  M\equiv -1\pmod P.
\]
\end{lemma}

\begin{proof}
Put
\[
  c=P-1-m.
\]
Since $P/p\leq m<P$, we have
\[
  0\leq c<P-P/p.
\]
The assumption $p\nmid m+1$ gives $p\nmid c$, because $m+1=P-c$.  Hence
$c>0$.

By Lemma~\ref{lem:lucas}, the congruence
$\binom{m+tM}{m}\not\equiv 0\pmod p$ is equivalent to the assertion that the
base-$p$ addition of $m$ and $tM$ has no carry.  Since $m<P$, this is
equivalent to
\[
  \res{tM}{P}\in\Dc_c.
\]
The hypothesis therefore says
\[
  \res{tM}{P}\notin\Dc_c
  \qquad(t=1,\ldots,m).
\]

Let $U$ be the inverse of $M$ modulo $P$.  Since $p\nmid c$, the element
$1$ belongs to $\Dc_c$.  Thus $U$ cannot lie in $\{1,\ldots,m\}$; also
$U\not\equiv 0\pmod P$.  Hence there is a unique $s\in\{1,\ldots,c\}$ such
that
\[
  U\equiv -s\pmod P.
\]

For any $y\in\Dc_c\setminus\{0\}$, the residue $\res{Uy}{P}$ cannot lie in
$\{1,\ldots,m\}$; otherwise, if $\res{Uy}{P}=t$ with $1\leq t\leq m$, then
$\res{tM}{P}=y\in\Dc_c$, a contradiction.  Therefore
\[
  \res{Uy}{P}\in\{m+1,\ldots,P-1\}
  =\{P-c,\ldots,P-1\}.
\]
Since $U\equiv -s\pmod P$, this is equivalent to
\[
  \res{s y}{P}\in\{1,\ldots,c\}
  \qquad(y\in\Dc_c\setminus\{0\}).
\]
By Theorem~\ref{thm:stabilizer}, and because $p\nmid c$, we get
$s\equiv 1\pmod P$.  Since $1\leq s\leq c<P$, this means $s=1$.  Hence
$U\equiv -1\pmod P$, and therefore $M\equiv -1\pmod P$.
\end{proof}

\subsection{The primary criterion}

\begin{lemma}
\label{lem:primary}
Let $n\geq 2$, let $p^a\Vert n$, and write
\[
  A=p^a,\qquad n=Ab,\qquad p\nmid b.
\]
Put
\[
  G_n=\gcd_{2\leq q\leq n}\binom{q(n-1)}{n-1}.
\]
Then
\[
  p\mid G_n
  \quad\Longleftrightarrow\quad
  b\leq A.
\]
\end{lemma}

\begin{proof}
Let $K=n-1=Ab-1$.

First suppose $p\nmid \binom{qK}{K}$.  By Lemma~\ref{lem:lucas}, the lower
$a$ base-$p$ digits of $qK$ must dominate the lower $a$ base-$p$ digits of
$K$.  But
\[
  K=Ab-1\equiv -1\pmod A,
\]
so those lower $a$ digits of $K$ are all $p-1$.  Hence the lower $a$ digits
of $qK$ are also all $p-1$, i.e.
\[
  qK\equiv -1\pmod A.
\]
As $K\equiv -1\pmod A$, this gives $q\equiv 1\pmod A$.

Thus any possible nonzero Lucas witness among $2\leq q\leq n=Ab$ has the
form
\[
  q=1+At,\qquad 1\leq t\leq b-1.
\]
For such $q$,
\[
  qK=(1+At)(Ab-1)=A\bigl(b+t(Ab-1)\bigr)-1.
\]
After the common lower $a$ digits, Lucas' criterion compares the base-$p$
digits of
\[
  b-1
  \quad\text{and}\quad
  b-1+t(Ab-1).
\]

Assume first that $b\leq A$.  For $1\leq t\leq b-1$,
\[
  b-1+t(Ab-1)\equiv b-t-1\pmod A,
\]
and $0\leq b-t-1<b-1\leq A-1$.  If the base-$p$ digits of $b-1$ were all
bounded by those of $b-1+t(Ab-1)$, then in particular the lowest $a$ digits
would represent a number at least $b-1$.  But those lowest $a$ digits
represent $b-t-1<b-1$.  Therefore Lucas' criterion fails for every
$t=1,\ldots,b-1$.  Hence every coefficient in the gcd is divisible by $p$,
so $p\mid G_n$.

Conversely, assume $p\mid G_n$.  Then for every $t=1,\ldots,b-1$,
\[
  \binom{(1+At)(Ab-1)}{Ab-1}\equiv 0\pmod p.
\]
Equivalently, after removing the common lower $a$ digits as above,
\[
  \binom{(b-1)+t(Ab-1)}{b-1}\equiv 0\pmod p
  \qquad(t=1,\ldots,b-1).
\]
Set
\[
  m=b-1,\qquad M=Ab-1.
\]
If $m=0$, then $b=1\leq A$ and there is nothing to prove.  Otherwise
$m\geq 1$, $p\nmid m+1=b$, and $p\nmid M$.  Let $P=p^L$ be the least
$p$-power with $m<P$.  By Lemma~\ref{lem:zero-run},
\[
  M\equiv -1\pmod P.
\]
Thus
\[
  Ab-1\equiv -1\pmod P,
\]
so $P\mid Ab$.  Since $p\nmid b$, this implies $P\mid A$.  Finally
$b=m+1\leq P\leq A$.  Therefore $b\leq A$.
\end{proof}

\subsection{Completion of the proof}

\begin{proof}[Proof of Theorem~\ref{thm:a080170}]
Let
\[
  D(k)=\gcd_{2\leq q\leq k+1}\binom{qk}{k},
  \qquad n=k+1.
\]
The gcd is a positive integer.  By Lemma~\ref{lem:prime-divides-n}, every
prime divisor of $D(k)$ divides $n$.

For a prime divisor $p$ of $n$, write $p^a\Vert n$, put $A=p^a$, and write
$n=Ab$.  Lemma~\ref{lem:primary}, with $G_n=D(k)$, gives
\[
  p\mid D(k)
  \quad\Longleftrightarrow\quad
  b\leq A
  \quad\Longleftrightarrow\quad
  \frac{n}{p^a}\leq p^a.
\]
Consequently,
\[
  D(k)=1
  \quad\Longleftrightarrow\quad
  \frac{n}{p^a}>p^a
  \text{ for every exact prime-power component }p^a\Vert n.
\]

It remains only to observe that this family of inequalities is equivalent to
the single inequality involving the largest exact prime-power component
$P=\ppart(n)$.  If $n/P>P$ and $A=p^a\Vert n$ is any other exact component,
then $A\leq P$, and hence
\[
  \frac{n}{A}\geq \frac{n}{P}>P\geq A.
\]
Conversely, if the inequality fails for some exact component $A$, then
$n/A\leq A\leq P$, so
\[
  \frac{n}{P}\leq \frac{n}{A}\leq P.
\]
Thus all component inequalities hold if and only if $n/P>P$.  This proves
the theorem.
\end{proof}

\section{The Lean formalization}
\label{sec:lean}

The theorem has also been formalized in Lean in
\texttt{FormalConjectures/OEIS/80170.lean} in the
\texttt{formal-conjectures} repository, under the statement
\texttt{OeisA80170.gcdCondition\_iff\_primePowerCondition}.  The file imports
\texttt{FormalConjectures.Util.ProblemImports} and
\texttt{Mathlib.Algebra.IsPrimePow}.  This section describes the organization
of that formal proof.  It is included because the Lean development is a
substantial part of the result: it fixes the exact formal meaning of the OEIS
statement, builds the auxiliary infrastructure used in the proof, and checks
the final assembly in the Lean kernel.  The natural-language proof in this
article and the Lean formalization were both generated by the MechMath Agent
Team \cite{mechmath}.

\subsection{Benchmark statement and compatibility layer}
\label{subsec:lean-target}

The Formal Conjectures file starts from the benchmark predicates rather than
from the notation used in Theorem~\ref{thm:a080170}.  In the displays below,
we use line-broken ASCII renderings of the Lean signatures; the linked source
contains the exact Unicode syntax.  The benchmark-facing statement is:

\begin{lstlisting}[style=lean]
def GCDCondition (k : Nat) : Prop :=
  (Finset.range k).gcd
    (fun i => Nat.choose ((i + 2) * k) k) = 1

def PrimePowerCondition (k : Nat) : Prop :=
  let P := ((Nat.divisors k).filter IsPrimePow).max.getD 0
  k / P > P

@[category research solved, AMS 11]
theorem gcdCondition_iff_primePowerCondition
    (k : Nat) (hk : 2 <= k) :
    GCDCondition k <-> PrimePowerCondition (k + 1)
\end{lstlisting}

Thus the formal theorem proves the equivalence between the gcd predicate at
\(k\) and the prime-power predicate at \(k+1\).  The informal proof, however,
is stated in terms of
\[
  D(k)=\gcd_{2\leq q\leq k+1}\binom{qk}{k}
  \qquad\text{and}\qquad n=k+1.
\]
The formal proof therefore introduces the paper-level definitions:

\begin{lstlisting}[style=lean]
def D (k : Nat) : Nat :=
  (Finset.Icc 2 (k + 1)).gcd
    fun q => Nat.choose (q * k) k

def ppart (n : Nat) : Nat :=
  n.primeFactors.sup
    fun p => p ^ n.factorization p

def digitBox (p L c : Nat) : Finset Nat :=
  (Finset.range (p ^ L)).filter
    fun y => forall i, i < L ->
      y / p ^ i % p <= c / p ^ i % p
\end{lstlisting}

The main theorem is obtained only after two compatibility bridges.  The first
bridge identifies the benchmark gcd over \texttt{Finset.range k} with the
paper's gcd over \texttt{Finset.Icc 2 (k + 1)}.  The second bridge identifies
the benchmark's maximum over prime-power divisors with the exact component
\(\ppart(n)=\max_{p^a\Vert n}p^a\).

\begin{lstlisting}[style=lean]
theorem gcdCondition_iff_D_eq_one (k : Nat) :
    GCDCondition k <-> D k = 1

theorem max_primePow_divisor_eq_ppart
    (n : Nat) (hn : 2 <= n) :
    ((Nat.divisors n).filter IsPrimePow).max.getD 0
      = ppart n

theorem primePowerCondition_iff_ppart
    (n : Nat) (hn : 2 <= n) :
    PrimePowerCondition n <-> n / ppart n > ppart n
\end{lstlisting}

These statements are mostly bookkeeping from a mathematical point of view,
but they are essential in Lean: the finite indexing set, the shift from \(k\)
to \(k+1\), the exact exponent \texttt{n.factorization p}, and the maximum
over prime-power divisors all have to be expressed as the same objects before
the final theorem can call the proof of Theorem~\ref{thm:a080170}.

\subsection{Formalizing the proof infrastructure}
\label{subsec:lean-infrastructure}

The formalization follows the proof section in two main preparatory blocks.
The first block proves the finite-difference reduction.  Its general tool is
\texttt{newton\_interpolation}, a Newton interpolation formula over
\(\mathbb Q\) for forward differences.  It is then specialized to the
polynomial attached to \(\binom{xk}{k}\), yielding the divisibility theorem
used in Lemma~\ref{lem:prime-divides-n}.

\begin{lstlisting}[style=lean]
theorem newton_interpolation
    (f : Polynomial Rat) (d : Nat) (hf : f.natDegree <= d)
    (a : Int) (x : Rat) :
    -- Newton expansion of f.eval x by forward differences at a

theorem prime_dvd_succ_of_dvd_D
    (k p : Nat) (hk : 1 <= k)
    (hp : p.Prime) (hdiv : p | D k) :
    p | k + 1
\end{lstlisting}

The second block formalizes the Lucas-theoretic criterion.  The paper states
this in base-\(p\) digit language; the Lean statement uses the explicit digit
formula \(N/p^i\bmod p\) at every position \(i\).

\begin{lstlisting}[style=lean]
theorem lucas_nonvanishing
    (p : Nat) (hp : p.Prime) (N K : Nat) :
    not (p | Nat.choose N K) <->
      forall i, K / p ^ i % p <= N / p ^ i % p
\end{lstlisting}

The proof calls Mathlib's formalized Lucas theorem inside this lemma.  After
\texttt{lucas\_nonvanishing} is available, congruence questions about
binomial coefficients can be converted into digit inequalities and then into
membership statements for \texttt{digitBox}.

\subsection{Digit boxes and the zero-run lemma}
\label{subsec:lean-digit-boxes}

The longest local development is the digit-box part.  The final stabilizer
statement in \texttt{80170.lean} is a multiplicative stabilizer theorem:
under the hypotheses of Theorem~\ref{thm:stabilizer}, multiplication by
\(s\) preserves the nonzero part of the digit box if and only if \(s\) is
congruent to \(1\) modulo the indicated power of \(p\).

\begin{lstlisting}[style=lean]
theorem digitBox_stabilizer
    (p L c s : Nat) (hp : p.Prime) (hL : 1 <= L)
    (hc0 : 0 < c) (hc : c < p ^ L - p ^ L / p)
    (hs : Nat.Coprime s p) :
    (forall y, y in digitBox p L c -> y != 0 ->
       1 <= s * y % p ^ L
       and s * y % p ^ L <= c)
    <->
    s == 1 [MOD p ^ (L - c.factorization p)]
\end{lstlisting}

The proof of this theorem is decomposed into smaller finite-set and digit
lemmas.  The basic membership and digit estimates are handled by
\texttt{mem\_digitBox}, \texttt{mem\_digitBox\_succ}, and
\texttt{le\_of\_mem\_digitBox}.  The gap and translation arguments are
isolated as \texttt{gap\_bound} and \texttt{no\_nonzero\_translation}; these
lemmas supply the finite combinatorial core needed by
\texttt{digitBox\_stabilizer}.

The zero-run lemma then turns a run of binomial divisibilities into a modular
conclusion.  Its formal statement is:

\begin{lstlisting}[style=lean]
theorem zero_run
    (p m M L : Nat) (hp : p.Prime) (hm : 1 <= m)
    (hM : Nat.Coprime M p) (hm1 : not (p | m + 1))
    (hLm : p ^ (L - 1) <= m) (hmL : m < p ^ L)
    (h : forall t, t in Finset.Icc 1 m ->
          p | Nat.choose (m + t * M) m) :
    (M : ZMod (p ^ L)) = -1
\end{lstlisting}

The conversions leading to \texttt{zero\_run} are also named explicitly in
the file:

\begin{lstlisting}[style=lean]
lemma noCarry_iff_mem_box
lemma not_dvd_choose_iff_no_carry
lemma not_dvd_choose_iff_mem_box
\end{lstlisting}

These lemmas are where the informal phrase ``no carry in base \(p\)''
becomes a statement about finite-set membership in \texttt{digitBox}.  This
extra layer is one of the reasons the formal proof is much longer than the
paper proof: every use of a residue modulo \(p^L\), every digit inequality,
and every conversion between a binomial congruence and a carry condition has
to be stated and typed explicitly.

\subsection{Primary criterion and final assembly}
\label{subsec:lean-assembly}

The local \(p\)-primary criterion is the central formal theorem connecting
the digit-box work back to the gcd:

\begin{lstlisting}[style=lean]
theorem primary_criterion
    (n p b : Nat) (hn : 2 <= n) (hp : p.Prime)
    (ha : 1 <= n.factorization p)
    (hb : n = p ^ n.factorization p * b)
    (hpb : not (p | b)) :
    p | D (n - 1) <-> b <= p ^ n.factorization p
\end{lstlisting}

The exponent in this statement is not a separate variable.  It is the exact
exponent \(\texttt{n.factorization p}\), so the hypotheses encode
\(n=p^a b\), \(a\geq 1\), and \(p\nmid b\) directly in Lean.  The proof uses
\texttt{lucas\_nonvanishing} for the easy direction and
\texttt{zero\_run} for the converse direction.

The paper theorem is then formalized as \texttt{a080170}:

\begin{lstlisting}[style=lean]
theorem a080170 (k : Nat) (hk : 2 <= k) :
    D k = 1 <->
      (k + 1) / ppart (k + 1) > ppart (k + 1)
\end{lstlisting}

Finally, the benchmark theorem is a short assembly step:

\begin{lstlisting}[style=lean]
@[category research solved, AMS 11]
theorem gcdCondition_iff_primePowerCondition
    (k : Nat) (hk : 2 <= k) :
    GCDCondition k <-> PrimePowerCondition (k + 1) := by
  rw [gcdCondition_iff_D_eq_one k, a080170 k hk,
    primePowerCondition_iff_ppart (k + 1) (by omega)]
\end{lstlisting}

This last block clearly shows the role of the compatibility layer: after the
benchmark predicates are rewritten into \(D(k)=1\) and
\((k+1)/\ppart(k+1)>\ppart(k+1)\), the kernel-checked proof is exactly the
formal theorem \texttt{a080170}.  The pinned source file contains no
\texttt{sorry}, and the final theorem is marked as a solved research
benchmark in the Formal Conjectures development.

\subsection{Remarks and availability}
\label{sec:remarks}

The complete Lean proof code is available at the pinned commit URL
\begin{center}
\url{https://github.com/guodk/formal-conjectures/blob/0720658844d76a50d48e4baa152eef14d4462907/FormalConjectures/OEIS/80170.lean#L1823}.
\end{center}
The corresponding pull request to the Formal Conjectures repository is
\begin{center}
\url{https://github.com/google-deepmind/formal-conjectures/pull/4253}.
\end{center}
The pinned code link is used to identify the exact artifact discussed in this
paper, while the pull request records the review context in the upstream
project.

\section{Conclusions}
\label{sec-conc}
In this paper, a binomial gcd criterion is proven. The criterion is listed as OEIS A080170 and appears as conjecture (17) in Ralf Stephan's list of OEIS conjectures.
Furthermore, the proof has been formalized in Lean, ensuring its correctness.
The problem under consideration is also a member of
\texttt{FC100OpenSet1}, a distinguished subset of 100 open research problems
in the Formal Conjectures benchmark. 
Our Lean proof has been accepted by the Formal Conjectures project.

\bibliographystyle{KLMM/klmm}
\bibliography{refs}

\end{document}